\documentclass[10pt,reqno,a4paper]{amsart}
\usepackage{amscd,amssymb,palatino}
\usepackage[utf8]{inputenc}
\usepackage[english]{babel}
\usepackage{caption}
\usepackage{etex}
\usepackage[leqno]{amsmath}%puts equation number on the left
\usepackage{amssymb}
\usepackage{color}
\usepackage{shadow}
\usepackage{epsfig}
\usepackage{epic}
\usepackage[dvipsnames]{xcolor}
\usepackage{graphics}
\usepackage{graphicx}
\usepackage{psfrag}
\usepackage{url}
\usepackage{hyperref}
\usepackage{calc}
\usepackage{tikz}
\usepackage{tikz-cd}
\usepackage[dvipsnames,prologue,table]{pstricks}
\usepackage{pstricks}
\usetikzlibrary{arrows,decorations,patterns,positioning,automata,shadows,fit,shapes,calc,decorations.markings,backgrounds,scopes,decorations.text}
\usepackage{fancyhdr}
\fancypagestyle{plain}{\fancyhf{}
\fancyfoot[C]{{\large\sf\bfseries\thepage}}}
\usepackage{calc}
\newcommand\1{\lower 9pt\hbox{\underbar{}}}
\numberwithin{equation}{section}

\newtheorem {theorem} {Theorem} [section]

\newtheorem {lemma}[theorem]{Lemma}
\newtheorem {prop}[theorem]     {Proposition}

\newtheorem {corollary}[theorem]       {Corollary}

\theoremstyle{definition}
\newtheorem {defi}[theorem]{Definition}
\newtheorem {Remark}[theorem]          {Remark}

\newtheorem* {theorem*}{Theorem}

\newtheorem {Example}[theorem]         {Example}

\newcommand{\pr} {\smallskip\noindent{\bf Proof\,\,}}

\def\R{\mathbb{R}}
\def\Z{\mathbb{Z}}

\newcommand{\pp}[2]{\frac{\partial#1}{\partial#2}}
\newcommand{\norm}[1]{\|#1\|}

\title{Constructing Turing complete Euler flows in dimension $3$}

\author{Robert Cardona}\address{ Robert Cardona,
Laboratory of Geometry and Dynamical Systems, Department of Mathematics, Universitat Polit\`{e}cnica de Catalunya and BGSMath Barcelona Graduate School of
Mathematics,  Avinguda del Doctor Mara\~{n}on 44-50, 08028 , Barcelona  \it{e-mail: robert.cardona@upc.edu }
 }
 \thanks{Robert Cardona acknowledges financial support from the Spanish Ministry of Economy and Competitiveness, through the Mar\'ia de Maeztu Programme for Units of Excellence in R\& D (MDM-2014-0445) via an FPI grant.}
\author{Eva Miranda}\address{ Eva Miranda,
Laboratory of Geometry and Dynamical Systems $\&$ Institut de Matemàtiques de la UPC-BarcelonaTech (IMTech), Department of Mathematics, EPSEB, Universitat Polit\`{e}cnica de Catalunya BGSMath Barcelona Graduate School of
Mathematics in Barcelona and
\\ IMCCE, CNRS-UMR8028, Observatoire de Paris, PSL University, Sorbonne
Universit\'{e}, 77 Avenue Denfert-Rochereau,
75014 Paris, France \it{e-mail: eva.miranda@upc.edu, eva.miranda@icmat.es }
 }
\thanks{Robert Cardona and Eva Miranda are partially supported by the grants MTM2015-69135-P/FEDER and PID2019-103849GB-I00 / AEI / 10.13039/501100011033, and AGAUR grant 2017SGR932. Eva Miranda is supported by the Catalan Institution for Research and Advanced Studies via an ICREA Academia Prize 2016.}
\author{Daniel Peralta-Salas} \address{Daniel Peralta-Salas, Instituto de Ciencias Matem\'aticas-ICMAT, C/ Nicol\'{a}s Cabrera, nº 13-15 Campus de Cantoblanco, Universidad Aut\'{o}noma de Madrid,
28049 Madrid, Spain \it{e-mail: dperalta@icmat.es} }
\thanks{Daniel Peralta-Salas is supported by the grants MTM PID2019-106715GB-C21 (MICINN) and Europa Excelencia EUR2019-103821 (MCIU)}
\author{Francisco Presas} \address{Francisco Presas, Instituto de Ciencias Matem\'aticas-ICMAT, C/ Nicol\'{a}s Cabrera, nº 13-15 Campus de Cantoblanco, Universidad Aut\'{o}noma de Madrid,
28049 Madrid, Spain \it{e-mail: fpresas@icmat.es} }
\thanks{Francisco Presas is supported by the grant reference number MTM2016-79400-P (MINECO/FEDER). This work was partially supported by the ICMAT--Severo Ochoa grant CEX2019-000904-S}

\begin{document}

\begin{abstract}
Can every physical system simulate any Turing machine? This is a classical problem which is intimately connected with the \emph{undecidability} of certain physical phenomena. Concerning fluid flows, Moore asked in~\cite{Mo} if hydrodynamics is capable of performing computations. More recently, Tao launched a programme based on the Turing completeness of the Euler equations to address the blow up problem in the Navier-Stokes equations. In this direction, the undecidability of some physical systems has been studied in recent years, from the quantum gap problem~\cite{perez} to quantum field theories~\cite{F}. To the best of our knowledge, the existence of undecidable particle paths of 3D fluid flows has remained an elusive open problem since Moore's works in the early 1990's. In this article we construct a Turing complete stationary Euler flow on a Riemannian $\mathbb S^3$ and speculate on its implications concerning Tao's approach to the blow up problem in the Navier-Stokes equations.
\end{abstract}

\maketitle

\section{Introduction}\label{S1}

In the book \emph{The Emperor's new mind}~\cite{penrose}  Roger Penrose returns to the artificial intelligence debate to convince us that creativity cannot be presented as the output of a ``mind'' representable as a Turing machine. This idea, which is platonic in nature and highly philosophical, evolves into more tangible questions such as: \emph{What kind of physics might be non-computational?}.

The ideas of the book are a source of inspiration and can be taken to several landscapes and levels of complexity: Is hydrodynamics capable of performing computations? (Moore~\cite{Mo}). Given the Hamiltonian of a quantum many-body
system, is there an algorithm to check if it has a spectral gap? (this is known as \emph{the spectral gap problem}, recently proved to be undecidable~\cite{perez}). And last but not least, can a mechanical system (including a fluid flow) simulate a universal Turing machine (\emph{universality})? (Tao~\cite{T1,T2, T3}).

This last question has been analyzed related to the conjecture of the regularity of the Navier-Stokes equations~\cite{T0}, which is one of the unsolved problems in the Clay's millennium list. In~\cite{TNat} Tao suggests a connection between a potential blow-up of the Navier-Stokes equations and Turing completeness and fluid computation. It is interesting to mention that another of the one million dollars problem on the same list whose resolution is still pending is the \emph{$P$ versus $NP$ problem}, which concerns the complexity of systems. Grosso modo, the question is if any problem whose solution can be  \emph{verified} by an algorithm polynomial in time (``of type $NP$'') can also be \emph{solved} by another algorithm polynomial in time (``of type $P$''). The delicate distinction between verification and solution has opened up an intricate scenery combining research in theoretical computer science, physics and mathematics. Although there is no apparent relation between these two celebrated problems, understanding a fluid flow as a Turing machine may shed some light on their connection.

On the other hand, undecidability of systems is everywhere and also on the invisible fine line between geometry and physics:
As proven by Freedman~\cite{F} non-abelian topological quantum field theories exhibit the mathematical features (combinatorics) necessary to support an NP-hard model. This relates topological quantum field theory and the Jones polynomial (as described by Witten~\cite{W}) to the $P\neq NP$ problem. Other undecidable problems on the crossroads of geometry and physics are the stability of an $n$-body system~\cite{Mo2}, the problem of finding an Einstein metric for a fixed $4$-fold as observed by Wolfram~\cite{Wo}, ray tracing problems in 3D optical systems~\cite{RTY}, or neural networks~\cite{S95}. Fundamental questions at the heart of low dimensional geometry and topology such as verifying the equivalence of two finitely specified 4--manifolds~\cite{Wo} or the problem of computing the genus of a knot~\cite{AHT} have also been proven to be undecidable and $NP$-hard problems, respectively.

In this article we address the appearance of undecidable phenomena in fluid dynamics proving the existence of Turing complete fluid flows on a Riemannian 3-dimensional sphere. Our novel strategy fusions the computational power of symbolic dynamics with fine techniques in contact topology and its connection with hydrodynamics unveiled by Sullivan, Etnyre and Ghrist more than two decades ago. The type of flows that we consider are stationary solutions to the Euler equations, which describe the dynamics of an inviscid incompressible flow in equilibrium. Its close companion, the Navier-Stokes equations, describe the dynamics of the viscid case. We end up this article discussing an application to these equations.
\\

%
% we prove that dimension three is enough to find a Turing complete Euler flow in the sphere, even though the metric is not the standard one. bla bla... the evolution of this solution in the corresponding Navier-Stokes equation, however, is not Turing complete since....
\textbf{Acknowledgements:} We are thankful to Robert Ghrist, Cristopher Moore, David Pérez-García and Leonid Polterovich for their feedback and suggestions on a former version of this article.

\section{Euler equations and Beltrami fields}\label{S.introEuler}
Euler equations model the dynamics of an incompressible fluid flow without viscosity. Even if they are classically considered on $\mathbb{R}^3$, they can be formulated on any $3$-dimensional Riemannian manifold $(M,g)$. The $\operatorname{div}$ and $\operatorname{curl}$ operators follow the same mnemonics as the classical ones in general coordinates from vector calculus (for an introduction to the topic see~\cite{AK,peralta}). The equations read as:
\begin{equation*}
\begin{cases}
\frac{\partial}{\partial t} X + \nabla_X X &= -\nabla p\,, \\
\operatorname{div}X=0\,,
\end{cases}
\end{equation*}
\noindent where $p$ stands for the inner pressure and $X$ is the velocity field of the fluid (a non-autonomous vector field on $M$). Here $\nabla_X X$ denotes the covariant derivative of $X$ along $X$.
%One of the important notions in the theory of Euler's equations is the concept of vorticity $\omega$ which is given by the generalization of curl operator.
%as the only vector $\omega:=\curl u$ satisfying,
%$$\iota_\omega \mu= (d\alpha)\,.$$
A solution to the Euler equations is called stationary whenever $X$ does not depend on time, i.e., $\frac{\partial}{\partial t} X=0$. As shown by the celebrated Arnold's structure theorem~\cite{AK}, among stationary solutions, Beltrami fields play a central role. We say that a divergence-free vector field $X$ on $(M,g)$ is Beltrami if
$$ \operatorname{curl}X=fX\,,$$
with $f \in C^\infty (M)$. When $f$ is non-vanishing, we call those vector fields \emph{rotational}. Hopf fields on $\mathbb S^3$ and ABC flows on $\mathbb T^3$ are examples of rotational (actually with constant factor) Beltrami fields.

The Euler equations can be defined in higher dimensions~\cite{AK}, an extension that is very useful to show that the steady Euler flows exhibit remarkable universality features. In~\cite{CMPP} we established a sort of universality of the stationary Euler flows by proving that any non-autonomous dynamics can be embedded into a steady Euler flow of high dimension. To achieve this, we used stationary solutions of Beltrami type (which can be defined on any odd dimensional manifold) and a remarkable connection with contact geometry~\cite{EG}: a correspondence principle between Beltrami fields and Reeb flows of contact structures in arbitrary dimension, which we describe in the next section. In our construction we also benefited from the $h$-principle in contact topology via novel Reeb embedding theorems which are key to establish the universality properties, see~\cite{CMPP} for details.
% have been widely used to exhibit very different topological and dynamical behavior of steady Euler flows. As a byproduct of our construction,  a Turing complete steady Euler flow was contructed on a 17-dimensional in the sense of \cite{T1}. Our construction relied on Tao's construction of Turing complete flows on a 4-torus \cite{T1}.  In this article, we sharpen this result considerably to obtain Turing complete Euler flows in dimension 3.

\section{Viewing fluid flows through a contact mirror}

A contact structure on an odd dimensional manifold $M^{2n+1}$ is determined by a hyperplane distribution $\xi$ given (at least locally) by the kernel of a one form $\alpha$ such that $\alpha\wedge (d\alpha)^n\neq 0$ (condition known as \emph{maximal non-integrability}). We will assume that the distribution is co-oriented (i.e., its normal bundle is oriented). This condition is equivalent to having a global one form defining the contact structure (called a defining contact form). For a fixed contact form we define its associated Reeb field $R$ by the equations $\alpha(R)=1$,  $\iota_{R}d\alpha=0$.
Contact geometry is often seen as the odd dimensional analogue of symplectic geometry. Indeed, symplectic and contact manifolds are related by several constructions. In particular, the contactization of an exact symplectic manifold $(M, d\lambda)$ is defined as the manifold $\R\times M$ equipped with the contact structure $\xi_{\lambda}= \ker(dt +\lambda)$. A key result in contact geometry is the existence of a Darboux theorem: the only local invariant of a contact structure is the dimension. The most simple proof is given by the following path method result in the contact realm:

\begin{theorem}[Gray stability theorem~\cite{Gray}]\label{thm:Gray}
Let $\xi_t$, $t\in [0,1]$, be a smooth homotopy of contact structures on a closed contact manifold $M$. Then there is an isotopy $\psi_t$ of $M$ such that $\psi_t^*\xi_0=\xi_t$ for each $t\in [0,1]$.
 \newline
Moreover, if the family is constant in the complement of a compact set $K$, then the diffeomorphisms $\psi_t$ are the identity away from $K$.
\end{theorem}

There is a surprising connection between contact geometry and hydrodynamics: in short, any non-vanishing Beltrami field can be reparametrized as the Reeb vector field of a contact structure. This geometrical discovery was suggested by Sullivan and proved by Etnyre and Ghrist~\cite{EG}.

In order to understand this remarkable correspondence it is convenient to rewrite the Euler equations in a \emph{dual language}. Duality is given by contraction with the Riemannian metric $g$. With the one form $\alpha$ defined as
$\alpha:=g(X,\cdot)$ and the Bernoulli function as $B:=p+\frac{1}{2}g(X,X)$, the steady Euler equations become,
\begin{equation*}
\begin{cases}
\iota_X d\alpha=-dB\,, \\
d\iota_X\mu=0\,,
\end{cases}
\end{equation*}
where $\mu$ is the Riemannian volume form.

Observe that:
\begin{itemize}

\item The equation $ \operatorname{curl}X=fX, \text{ with } f \in C^\infty (M) $, satisfied by a Beltrami vector field on a $3$-manifold, can be equivalently written as {$ d\alpha = f \iota_X \mu\,. $} Assume that $X$ is rotational, e.g., $f>0$, then if $X$ does not vanish on $M$ we infer that

$$\alpha \wedge d\alpha= f \alpha \wedge \iota_X \mu > 0\,,$$

thus proving that $\alpha$ defines a contact structure on $M$.

\item Obviously, $X$ satisfies {$\iota_X (d\alpha)= \iota_X \iota_X \mu=0$}, so $X \in \ker d\alpha$  and therefore it is a reparametrization of the Reeb vector field by the function $\alpha(X)= g(X,X)$, i.e,. $R=\frac{X}{\alpha(X)}$.
\end{itemize}

These observations prove one of the implications of the following theorem, which is due to Etnyre and Ghrist~\cite{EG}. This result will be a key instrumental tool to construct a Turing complete Euler flow on a Riemannian $\mathbb S^3$ in Section~\ref{S.Euler}.

\begin{theorem}\label{correspondence}
Any non-vanishing rotational Beltrami field is a reparametrization of a Reeb vector field for some contact form. Conversely, any reparametrization of a Reeb vector field of a contact structure is a non-vanishing rotational Beltrami field for some Riemannian metric and volume form.
\end{theorem}

Beyond this characterization of (non-vanishing, rotational) Beltrami fields in terms of contact geometry, a natural question is whether a general volume preserving vector field is a solution to the Euler equations for some Riemannian metric. More concretely, to address this problem we introduced in~\cite{PRT} the following definition:

\begin{defi}
Let $M$ be a (not necessarily closed) manifold endowed with a volume form $\mu$. A volume preserving vector field $X$ is \emph{Eulerisable} if there is a metric $g$ on $M$ for which $X$ satisfies the stationary Euler equations for some Bernoulli function $B:M \rightarrow \R$.
\end{defi}

Since the Eulerisable fields are stationary solutions of the Euler equations on some $(M,g)$, they exist for all time. With this definition one gets rid of the metric (which is no longer fixed as in the standard setting of the Euler equations), which allows one to exploit the enormous geometric wealth of the Euler equations. The contact mirror stated in Theorem~\ref{correspondence}, which also holds on manifolds with boundary~\cite{CMP}, is just the first instance of a series of striking results that establish connections with geodesible flows~\cite{Re,CV} or Sullivan's theory of foliated cycles~\cite{PRT}. Moreover, as mentioned in Section~\ref{S.introEuler}, the Eulerisable flows have proved to be flexible enough to encode any non-autonomous dynamics by increasing the dimension of the ambient manifold (``universality''), see~\cite{CMPP}.
%\begin{theorem}[\cite{CMPP}]\label{T.main1}
%The Euler flows are universal. Moreover, the dimension of the ambient manifold $\mathbb S^n$ or $\mathbb R^n$ is the smallest odd integer~$n\in\{3\text{ dim }N+5,3\text{ dim }N+6\}$.
%In the time-periodic case, the extended field $u$ is a steady Euler flow with a metric $g=g_0+\delta_P$, where $g_0$ is the canonical metric on $\mathbb S^n$ and $\delta_P$ is supported in a ball that contains the invariant submanifold $e(N\times\mathbb S^1)$.
%\end{theorem}

\section{An excursion to contact topology}\label{S:contact}

The goal of this section is to prove the following result in contact topology, which is a key ingredient for the proof of the main result in this article. All along this section $D_{\rho}$ is a $2$-dimensional disk of radius $\rho$. If $\rho=1$ we just omit it to write $D$.

\begin{theorem}\label{prop:Reebdisk} Let $(M,\xi)$ be a contact $3$-manifold and $\varphi:D \rightarrow D$ an area-preserving diffeomorphism of the disk which is the identity (in a neighborhood of) the boundary. Then there exists a defining contact form $\alpha$ whose associated Reeb vector field $R$ exhibits a Poincar\'e section with first return map conjugated to $\varphi$.
\end{theorem}

Combining this result with the contact/Beltrami correspondence (Theorem~\ref{correspondence}) we obtain a metric $g$ on $M$ for which $R$ is a Beltrami field. This yields the following corollary:

\begin{corollary}\label{prop:BeltDisk}
Let $M$ be a $3$-manifold. Then, given any area-preserving diffeomorphism $\varphi:D\rightarrow D$ of the disk which is the identity (in a neighborhood of) the boundary, there exists a metric $g$ on $M$ such that $\varphi$ can be realized as the first return map of some Beltrami field $X$ on $(M,g)$, up to conjugation.
\end{corollary}
\begin{Remark}
To prove Corollary~\ref{prop:BeltDisk} we need to use the well known fact (since the works of Martinet) that any 3-manifold admits a contact structure. In higher dimensions only almost contact manifolds admit contact structures and thus the existence of contact structures on a given manifold is topologically obstructed~\cite{BEM}.
\end{Remark}

\begin{proof}[Proof of Theorem~\ref{prop:Reebdisk}]

\hfill \newline
We divide the proof in two steps. The first one realizes the diffeomorphism $\varphi$ as the first-return map of a Reeb vector field on a solid torus. This result is not new, but we provide an alternative (and simpler) argument to the proof presented by Bramham in~\cite[Chapter 4]{Bra}. In the second step we globalize the previous construction to obtain a Reeb field on any $3$-manifold.

\hfill \newline
\textbf{Step 1: Constructing a Reeb mapping torus.}

Let us denote by $\lambda$ the one form $r^2d\phi$, where $(r,\phi)$ are polar coordinates on the disk $D$. In particular, the form $dt+\lambda$ on $D\times [0,1]$ ($t$ is the coordinate on $[0,1]$) defines a contact form. Since the diffeomorphism $\varphi$ is area-preserving and the identity in a neighborhood of $\partial D$, it is isotopic to the identity and the time-1 flow of a Hamiltonian (non-autonomous) vector field. More precisely, there is a family of diffeomorphisms $\varphi_t$, which are the identity in a neighborhood of $\partial D$ for all $t\in[0,1]$, such that $\varphi_1=\varphi$, $\varphi_0=\operatorname{id}$, and this family is generated by a family of compactly supported vector fields $X_t$ so that
\begin{equation} \label{eq:Ham}
\iota_{X_t}d\lambda=dH_t\,,
\end{equation}
where $H_t$ is a family of functions (Hamiltonians) of the disk. For each $t\in[0,1]$, $H_t$ is obviously constant on a neighborhood of $\partial D$. Additionally, we can safely assume that $\varphi_t$ is the identity for $t<\delta$ and $t>1-\delta$, which implies that $H_t$ is constant on $D$ ($t$-dependent) for $t<\delta$ and $t>1-\delta$. Accordingly, redefining $H_t$ if necessary, we can assume that $H_t=0$ in a neighborhood of $\partial D$ for all $t\in[0,1]$, and $H_t=0$ on the whole $D$ for $t<\delta$ and $t>1-\delta$. Let us now define the function
\begin{align*}
\widetilde H: D\times[0,1] &\longrightarrow \mathbb{R} \\
		(a,z) &\longmapsto \tilde H(a,z):=H_z(a)\,,
\end{align*}
and the one form
$$\widetilde \alpha:=(\widetilde H+C)dz + \lambda$$
on the cylindrical set $D\times [0,1]$, where $C$ is a positive constant. We claim that for a large enough constant $C$, $\widetilde \alpha$ is a contact form. Indeed, a straightforward computation shows that

\begin{equation}\label{eqn:C}
 \widetilde \alpha \wedge d\widetilde \alpha= Cd\lambda\wedge dz+\widetilde Hd\lambda\wedge dz-d\widetilde H\wedge \lambda\wedge dz\,,
 \end{equation}

which is obviously positive if $C>C_0$, a constant that only depends on the $C^1$-norm of $\widetilde H$ on $D$. Additionally, the Reeb field of this contact form is a multiple of $\pp{}{z}+X$, where $X$ is defined as $X (a,z):=X_z(a)$. This is equivalent to checking the condition $\iota_{\pp{}{z}+X}d\widetilde \alpha=0$. Indeed, we have
\begin{align*}
\iota_{\pp{}{z}+X}d\widetilde \alpha&=\iota_{\pp{}{z}+X} (d\lambda + dH_z \wedge dz) \\
&= \iota_Xd\lambda-dH_z + (\iota_X dH_z)dz=0\,,
\end{align*}
where we have used Equation~\eqref{eq:Ham} to cancel out the first two terms in the second equality, and the third summand in the equality vanishes by contracting the same equation with $X$. In particular, the flow of the Reeb vector field of $\widetilde \alpha$ is a reparametrization of the flow of $\pp{}{z} + X$, whose time-one map is given by $\varphi$.

By the construction of the Hamiltonian family $H_t$, we conclude that the contact form $\widetilde \alpha$ is equal to $Cdz + \lambda$ on a neighborhood of the boundary of the set $D\times [0,1]$, and it descends to the quotient (the solid torus $D \times \mathbb S^1$, where the coordinate $z$ goes to a coordinate $\theta$ in $\mathbb S^1$). Still denoting the contact form in the quotient as $\widetilde \alpha$, it is obvious that near the boundary of the solid torus, $\widetilde \alpha$ is $Cd\theta+r^2d\phi$.

%\textcolor{magenta}{Moreover, realize that this construction is clearly parametric and we can apply it to the family $\varphi^s_{t}= \varphi_{s \cdot t}$, where $s$ is fixed and we are just constructing the flow from time $0$ to $s$. From this, the Hamiltonian function is $H(\phi_{s,t})= s\cdot  H_{ts}$. So its $C^1$-norm decreases and the constant $C>0$, according to equation \ref{eqn:C}, can be chosen uniform for the whole family: the value for $s=1$ works for all the values $s\in[0,1]$. This yields a family $\widetilde{\alpha}_s$ whose diffeomorphisms are defined on the same domain.}

%By the construction of the Hamiltonian family $H_t$, we conclude that the contact form $\widetilde \alpha$ is equal to $Cdz + \lambda$ on a neighborhood of the boundary of the set $D\times [0,1]$, and it descends to the quotient (the solid torus $D^2 \times S^1$, where the coordinate $z$ goes to a coordinate $\theta$ in $S^1$). Still denoting the contact form in the quotient as $\widetilde \alpha$, it is obvious that near the boundary of the solid torus, $\widetilde \alpha$ is $Cd\theta+r^2d\phi$.

\hfill \newline

\textbf{Step 2: Global extension.}
%
%Any homotopy class of non-vanishing vector field of any three manifold admits a contact structure .

Let $(M,\xi)$ be a contact $3$-manifold and take a circle $\gamma$ transverse to the contact structure (which always exists and can be chosen $C^0$ close to any given closed curve). It is standard that there are coordinates $(r',\phi,\theta)$ in a neighborhood $U=D_\rho \times S^1$ of the circle $\gamma=\{0\} \times S^1$ such that $\xi$ is defined by the kernel of the contact form
$$\beta_0=C(d\theta + r'^2d\phi)\,. $$
Here $D_\rho$ is a $2$-dimensional disk of small enough radius $\rho$, the coordinates are the standard angle $\theta$ on $\mathbb S^1$ and polar coordinates $(r',\phi)$ on $D_\rho$, and $C$ is the large constant introduced in Step~1. In particular, multiplying by a suitable positive factor if necessary, we can take a global contact form $\beta$ defining $\xi$ such that $\beta|_U=\beta_0$. Now we observe that the contact form $\widetilde \alpha$ obtained in Step~1 can be constructed on a disk $D_\rho$ of arbitrary radius using a $\rho$-rescaling of $D$ (a conjugation): $\Phi_\rho:D\times \mathbb S^1\to D_\rho\times\mathbb S^1$, with $(r',\phi,\theta):=\Phi(r,\phi,\theta)=(\rho r,\phi,\theta)$. Since $C>0$ is any large enough constant, it is clear that we can take the radius $\rho$ to be $$\rho=C^{-1/2}\,.$$
The (large) constant $C$ is fixed in what follows.

Using the conjugation $\Phi_\rho$ we define the contact form $\widetilde\alpha':=\Phi_{\rho*}\widetilde \alpha$ on $U$. Specifically,
\[
\widetilde\alpha'=(\widetilde H'+C)d\theta+Cr'^2d\phi\,,
\]
where $\widetilde H'\equiv \widetilde H'(r',\phi,\theta)=\widetilde H(C^{1/2}r',\phi,\theta)$. Obviously the associated Reeb field is a multiple of $\Phi_{\rho*}\Big(\pp{}{z}+X\Big)$; in particular, its first return map on the section $D_\rho\times \{0\}$ is given by $\Phi_\rho\circ\varphi\circ \Phi_\rho^{-1}$. Additionally, by construction,
\[
\widetilde\alpha'= C(d\theta+r'^2d\phi)
\]
in a neighborhood of the boundary of the solid torus $U$. $\widetilde\alpha'$ can then be extended to the complement $M\backslash U$ as $\beta$ because $\widetilde\alpha'=\beta_0$ on $\partial U$. In summary, denoting by $\alpha$ this globally defined contact form on $M$, we conclude that it coincides with $\widetilde \alpha'$ in $U$, and its associated Reeb field in this set has a first return map that is conjugated to $\varphi$.

It remains to show that the contact structure defined by $\ker \alpha$ is homotopic through contact structures to $\xi$. Indeed, let us define the family of one forms
\[
a_t:=((1-t)\widetilde H'+C)d\theta+Cr'^2d\phi
\]
in the toroidal set $U$. A straightforward computation shows that $a_t$ is a contact form for all $t\in[0,1]$. Moreover, $a_0=\alpha$ and $a_1=C(d\theta+r'^2d\phi)=\beta_0$. Noticing that $\alpha=\beta$ in $M\backslash U$, this yields a global homotopy of contact forms that interpolates $\alpha$ with $\beta$, which immediately implies an homotopy of contact structures $\xi_t$ such that $\xi_0=\ker \alpha$ and $\xi_1=\xi$. Applying Theorem~\ref{thm:Gray}, we deduce that $\ker \alpha$ is contactomorphic to $\xi$. The theorem then follows.

\end{proof}%
%
%{\red The following theorem \cite{Lut, Mar} is classical in contact topology.
%\begin{theorem}[Lutz-Martinet]
%Let $M$ be a three dimensional closed manifold. In each homotopy class of hyperplane fields, there is a contact structure.
%\end{theorem}
%Since homotopy classes of hyperplane fields and their  line fields are in bijection, we deduce that the Beltrami field constructed in Theorem \ref{thm:main} can be taken in any homotopy class of non-vanishing vector fields. }

\begin{Remark}\label{Rem:fixcont}
An easy modification of the proof of Step~2 allows us to choose the defining contact form $\beta$ in the complement of the toroidal set $U$. More precisely, given any contact form $\beta$ defining the contact structure $\xi$, there is another defining contact form $\alpha$ such that $\alpha=\widetilde\alpha'$ on $U$ and $\alpha=\beta$ in the complement of a neighborhood of $U$.
\end{Remark}

%\textcolor{magenta}{\begin{Remark}
%The connectedness of the space of contact structures in the $3$-ball relative to the boundary  \cite{Eliash} lets us avoid the parametric construction, i.e. to define the family of contact forms $\hat{\alpha}_s$ interpolating  $\hat{\alpha}_0= \iota^* \beta_0$ and $\hat{\alpha}_1$. The path  $\hat{\alpha}_s$ exists by \cite{Eliash} so there is no need to create it.
%\end{Remark}
A similar proof provides an equivalent statement for higher dimensional contact manifolds. It reads as follows:
\begin{theorem}
Let $(M,\xi)$ be a contact $(2n+1)$-manifold and $\varphi:D \rightarrow D$ a symplectomorphism of the $2n$-ball which is the identity (in a neighborhood of) the boundary. Then there exists a defining contact form $\alpha$ whose associated Reeb vector field $R$ exhibits a Poincar\'e section with first return map conjugated to $\varphi$.
\end{theorem}

\section{Turing machines and symbolic dynamics}\label{S:Turing}

The key tool to construct a dynamical system that simulates a Turing machine is symbolic dynamics. Our goal in this section is to recall some basic properties of Turing machines and to introduce Moore's theory~\cite{Mo} on the connection between diffeomorphisms of manifolds and computation. In particular, we shall show that suitable generalizations of the shift map are enough to simulate universal Turing machines. This paves the way to construct a Turing complete area-preserving diffeomorphism of the disk, as we shall see in Section~\ref{S:diffeo}.

\begin{Remark}
As discussed in several parts of the literature~\cite{DKB, Del}, there are misleading intuitions that lead to conclude that the shift map can simulate a universal Turing machine. This happens when we accept to take initial points that are not constructible, i.e., they contain as initial information all the computations of the Turing machine instead of just its initial tape. See Section~\ref{SS.genshift} for a detailed explanation.
\end{Remark}

\subsection{Turing machines}\label{SS.Turing}

A Turing machine is defined via the following data:

\begin{itemize}
\item A finite set $Q$ of ``states'' including an initial state $q_0$ and a halting state $q_{halt}$.
\item A finite set $\Sigma$ which is the ``alphabet'' with cardinality at least two.
\item A transition function $\delta:(Q\times \Sigma) \longrightarrow (Q\times \Sigma \times \{-1,0,1\})$.
\end{itemize}

Let us denote by $q\in Q$ the current state, and by $t=(t_n)_{n\in \mathbb{Z}}\in \Sigma^\mathbb{Z}$ the current tape. For a given Turing machine $(Q,q_0,q_{halt},\Sigma,\delta)$ and an input tape $s=(s_n)_{n\in \mathbb{Z}}\in \Sigma^{\mathbb{Z}}$ the machine runs applying the following algorithm:

\begin{enumerate}
\item Set the current state $q$ as the initial state and the current tape $t$ as the input tape.
\item If the current state is $q_{halt}$ then halt the algorithm and return $t$ as output. Otherwise compute $\delta(q,t_0)=(q',t_0',\varepsilon)$, with $\varepsilon \in \{-1,0,1\}$.
\item Replace $q$ with $q'$ and $t_0$ with $t_0'$.
\item Replace $t$ by the $\varepsilon$-shifted tape, then return to step $(2)$. Following Moore~\cite{Mo}, our convention is that $\varepsilon=1$ (resp. $\varepsilon=-1$) corresponds to the left shift (resp. the right shift).
\end{enumerate}

In particular, the space of all possible internal states of a Turing machine is given by $\mathcal{P}:=Q\times \Sigma^\mathbb{Z}$. The transition function $\delta$ induces a global transition function $\Delta: Q\setminus \{q_{halt} \} \times \Sigma^{\mathbb{Z}} \rightarrow \mathcal{P}$, which sends an internal state in $\mathcal{P}$ to the internal state obtained after applying a step of the algorithm.

It is convenient to decompose the transition function $\delta$ in three components that we can denote by $F_1,F_2,F_3$:
\begin{align*}
F_1: Q\times \Sigma &\longrightarrow Q  \\
F_2: Q\times \Sigma &\longrightarrow \Sigma \\
F_3: Q\times \Sigma &\longrightarrow \{-1,0,1\}.
\end{align*}

The first component $F_1$ tells you the new state $q'$ in terms of the current state $q$ and the tape value $t_0$. The second component $F_2$ computes the new value of the tape cell $t_0'$ in terms of the state $q$ and the tape value $t_0$. Finally, the last component $F_3$ tells you if the tape should stay, shift to the left or shift to the right in terms of the current state and the tape value $t_0$.

As we shall see in Section~\ref{SS.genshift}, the reversible condition plays a crucial role to construct well behaved dynamical systems that simulate universal computation. Reversibility of a Turing machine can be defined in several equivalent ways. The definition in terms of the global transition function $\Delta$ is suitable for our discussion.

\begin{defi}\label{D:revers}
A Turing machine $T=(Q,q_0,q_{halt},\Sigma,\delta)$ is \emph{reversible} if the global transition function $\Delta$ is injective.
\end{defi}

\subsection{Generalized shifts and Turing machine simulation}\label{SS.genshift}

In this section we introduce Moore's theory of generalized shifts~\cite{Mo}, which will be instrumental to construct a Turing complete area-preserving diffeomorphism of the disk in Section~\ref{S:diffeo}.

First, an important remark on the necessity of this theory is in order: let us elaborate on the reason why the shift map is not suitable to perform universal computation. Indeed, consider a Turing machine with some given initial input. We can associate to it a sequence $(q_i)_{i\in\mathbb{N}}$ of states, where $q_i$ is the state of the machine at the step $i$. If the machine reaches the halting state at a step $j$, we define $q_i=q_{halt}$ for all $i\geq j$. Using this sequence, we can construct another sequence $(p_i)_{i\in \mathbb{Z}} \in \{0,1\}^{\mathbb{Z}}$ by setting $p_i=0$ if $i<0$, $p_i=0$ if $i\geq 0$ and $q_i\neq q_{halt}$ and $p_i=1$ if $q_i=q_{halt}$. Iterating this sequence by the standard (left) shift map, it is clear that the Turing machine halts if and only if the shift map finds a digit $1$ in position $0$ at some iteration. The problem is, however, that the initial sequence is not constructible; it is a priori an undecidable problem to construct the whole sequence $(q_i)_{i\in \mathbb N}$.

To introduce the notion of generalized shift, let $A$ be an alphabet and $S\in A^\mathbb{Z}$ an infinite sequence. A generalized shift is specified by two maps $F$ and $G$ which depend on $\emph{finitely many}$ positions of $S$. Denote by $D_F= \{i,...,i+r-1\}$ and $D_G=\{j,...,j+l-1\}$ the sets of positions on which $F$ and $G$ depend, respectively. They have cardinality $r\geq 1$ and $l\geq 1$, respectively. Obviously, these functions take a finite number of different values since they depend on a finite number of positions. The function $G$ modifies the sequence only at the positions indicated by $D_G$:
\begin{align*}
G:A^l &\longrightarrow A^l \\
(s_{j}...s_{j+l-1}) &\longmapsto (s_{j}'...s_{j+l-1}')
\end{align*}
Here $s_j...s_{j+l-1}$ are the symbols at the positions $j,...,j+l-1$ of an infinite sequence $S\in A^\mathbb{Z}$.

On the other hand, the function $F$ assigns to the finite subsequence $(s_{i},...,s_{i+r-1})$ of the infinite sequence $S\in A^\mathbb{Z}$ an integer:
$$ F:A^{r}\longrightarrow \mathbb{Z} $$

A generalized shift $\phi:A^\Z \rightarrow A^\Z$ is then defined as follows:
\begin{itemize}
\item Compute $F(S)$ and $G(S)$.
\item Modify $S$ changing the positions in $D_G$ by the function $G(S)$, obtaining a new sequence $S'$.
\item Shift $S'$ by $F(S)$ positions. That is, we obtain a new sequence $s''_n=s'_{n+F(S)}$ for all $n\in \Z$.
\end{itemize}
The sequence $S''$ is then $\phi(S)$. For example, the standard shift is obtained by taking $G$ to be the identity and $F\equiv 1$. For later convenience, when taking a sequence in $A^\mathbb{Z}$, we will write a point to denote that the symbol at the right of that point is the symbol at position $0$. In particular, the sequence $(s_n)$ can be denoted by $(...s_{-1}.s_0s_1...)$.

The remarkable property of generalized shifts is that they can simulate any Turing machine in the following sense:

\begin{defi}
We say that a generalized shift $\phi$ with alphabet $A$ is conjugated to a Turing machine $T=(Q,q_0,q_{halt},\Sigma,\delta)$ if there is an injective map
$$\varphi: \mathcal{P} \rightarrow A^\mathbb{Z}$$
such that the global transition function of the Turing machine is given by
\begin{equation} \label{eq:conju}
\Delta= \varphi^{-1}\phi\varphi\,.
\end{equation}
We recall that $\mathcal{P}=Q\times \Sigma^\mathbb{Z}$ is the space of all possible internal states of $T$.
\end{defi}

The following result, which was first proved in~\cite{Mo}, establishes that any Turing machine is conjugated to a generalized shift. Although this result is relatively standard, we include a proof because it helps to elucidate the connection between Turing machines and Generalized Shifts.

% we include a proof because the explicit construction of the conjugation $\varphi$ will be used in the proof of Lemma~\ref{lem:revTur}.

\begin{lemma}\label{lem:TuringGS}
Given a Turing machine $T=(Q,q_0,q_{halt},\Sigma,\delta)$, there is a generalized shift conjugated to it.
\end{lemma}

\begin{proof}
Recall that $F_1,F_2,F_3$ denote the three components of the transition function $\delta$ of the Turing machine $T$. Let us construct a generalized shift whose alphabet is given by $A:=\Sigma \cup Q$, i.e., both the alphabet and the set of states of the Turing machine $T$. First, notice that to every internal state $(q,(t_i)_{i\in\Z})\in \mathcal{P}$ of $T$, we can assign the sequence $(\dots t_{-1}.qt_0t_1\dots)$ in $A^{\mathbb{Z}}$.

Now let us define the maps $F$ and $G$, which will depend only on the three positions $-1,0,1$, i.e., $D_F=D_G=\{-1,0,1\}$. For a sequence $(s_n)\in A^\mathbb{Z}$, denote by $a:=(s_{-1}s_0s_{1})$ the subsequence of symbols in positions $-1,0,1$. If this subsequence $a$ is not of the form $(t_{-1}qt_0)$ with $t_{-1},t_0 \in \Sigma$ and $q\in Q$, then we define $F(a):=0$ and $G(a):=a$. Otherwise, we have $a=t_{-1}qt_{0}$ for some symbols $t_{-1},t_{0}$ in $\Sigma$ and $q\in Q$. Setting $q':=F_1(q,t_0)$ and $t'_0:=F_2(q,t_0)$, we can define $F$ and $G$ as:

$$F(a):=F_3(q,t_0)\,,$$

and

\begin{equation}\label{Gturing}
G(a):=\begin{cases}
t_{-1}.t'_0q' \text{ if } F(a)=1\,,\\
q'.t_{-1}t'_0\text{ if } F(a)=-1\,,\\
t_{-1}.q't'_0 \text{ if } F(a)=0\,.
\end{cases}
\end{equation}

Although the function $F$ only depends on the positions $a_0$ and $a_1$ of $a$, we have chosen $D_F=D_G=\{-1,0,1\}$ so that both domains are the same. These maps $F$ and $G$ define a generalized shift $\phi$ as explained above.

Finally, given an internal state $(q,(t_i)_{i\in\Z}) \in \mathcal{P}$, it is straightforward to check that a step of the Turing machine $T$ corresponds to $\varphi^{-1}\phi \varphi$, where
\begin{align*}
\varphi: \mathcal{P} &\longrightarrow A^{\mathbb{Z}} \\
(q,(t_i)_{i\in\Z}) &\longmapsto s=\dots t_{-1}.qt_0t_1\dots
\end{align*}
is an injective map. By definition, this proves that $T$ is conjugated to $\phi$, and the lemma follows.
\end{proof}

The main property that we introduce in this section is that a reversible Turing machine is conjugated to a bijective generalized shift. This fact was stated in~\cite{Mo} without a proof. This is not immediately clear, since the global transition function of a Turing machine is not defined for the halting state $q_{halt}$, which prevents to extend the generalized shift when the symbol in position $0$ is $q_{halt}$. In order to fix this issue, and to obtain a bijective generalized shift via Lemma~\ref{lem:TuringGS}, we can extend the transition function by setting
\begin{equation}\label{eq:extens}
 \delta(q_{halt},t_0):=(q_0,t_0,0)\,,
\end{equation}
which ensures that $\delta$ is defined in all the domain of states $Q$. This also guarantees that $\Delta$ extends to all $\mathcal{P}$ in an injective way. When necessary, we shall assume that the global transition function has been extended this way without further mention. Moreover,  Equation~\eqref{eq:conju} is satisfied on the whole domain $\mathcal{P}$.

\begin{lemma}\label{lem:revTur}
A reversible Turing machine $T$, whose transition function has been extended as above, is conjugated to a bijective generalized shift.
\end{lemma}

\begin{proof}
By the previous discussion, the global transition function $\Delta$ of $T$ is injective and defined on the whole domain of states, so the conjugation specified in Equation~\eqref{eq:conju} also holds on the whole domain. Accordingly, the generalized shift map $\phi$ is injective when restricted to the subset $\varphi(\mathcal{P})\subset A^{\Z}$ because it is conjugated to the injective map $\Delta$; and it is also injective on its complement set, where it is the identity map. This shows that the generalized shift $\phi$ is injective, and in fact bijective by~\cite[Lemma~2]{Mo}, which completes the proof of the lemma.
\end{proof}

As shown by Moore, the relevance of generalized shifts comes from the fact that they are conjugated to maps of the square Cantor set, which allows one to use the machinery of symbolic dynamics (compare with the particular case of the standard shift map). We recall the following:

\begin{defi}
The square Cantor set is the product set $C^2:=C\times C \subset I^2$, where $C$ is the (standard) Cantor ternary set in the unit interval $I=[0,1]$. Additionally, we say that a Cantor block is a block of the form $B=[\frac{a}{3^i},\frac{a+1}{3^i}]\times [\frac{b}{3^j},\frac{b+1}{3^j}]$, where $i,j$ are nonnegative integers and $a<3^i$, $b<3^j$ are nonnegative integers such that there are points of $C^2$ in the interior of $B$.
\end{defi}

It is clear that for given $i,j$ we can find a finite amount of disjoint Cantor blocks whose union contains all the points of the square Cantor set. In what follows, we shall consider generalized shifts with alphabet $A=\{0,1\}$. Actually, as proved in~\cite[Lemma 1]{Mo}, this can always be assumed. Given an infinite sequence $s=(...s_{-1}.s_0s_1...)\in A^\Z$, we can associate to it an explicitly constructible point in the square Cantor set. The usual way to do this is to express the coordinates of the assigned point in base $3$: the coordinate $y$ corresponds to the expansion $(y_1,y_2,...)$ where $y_i=0$ if $s_i=0$ and $y_i=2$ if $s_i=1$. Analogously, the coordinate $x$ corresponds to the expansion $(x_{1},x_{2},...)$ in base $3$ where $x_i=0$ if $s_{-i}=0$ and $x_i=2$ if $s_{-i}=1$.

%{\red Se puede comentar que en particular la expansion trinaria de las coordenadas no contiene unos.. }
%
%{\red ADD FIGURE CANTOR SET? (maybe that's already too much figures)
%}

The aforementioned assignment between infinite sequences and points in the square Cantor set is key to prove a fundamental lemma that we borrow from Moore's work~\cite{Mo}. Combined with Lemma~\ref{lem:revTur} it will be key to construct a Turing complete area-preserving diffeomorphism of the disk in Section~\ref{S:diffeo}.

\begin{lemma}[Moore]\label{lem:Cantormap}
Any generalized shift is conjugated to the restriction to the square Cantor set of a piecewise linear map of $I^2$. This map consists of $k$ finitely many area-preserving linear components defined on Cantor blocks, with $k$ bounded as:
$$ k\leq n^{|D_F\cup D_G| + max|F|}\,.$$
Here $n:=|A|$. If the generalized shift is bijective, then the image blocks are pairwise disjoint.
\end{lemma}
\begin{Remark}
In the literature, there have been other attempts to simulate a reversible Turing machine by means of selecting the space of states of the machine as a constructible choice of coordinates in the square $I^2$ and extending the global transition function from that set of points to a bijective map of $I^2$. However, these other models do not provide a continuous~\cite{RTY} or a compactly supported extension~\cite{portug} or they increase the dimension~\cite{T1}. In Section~\ref{S:diffeo} we will show that Moore's approach has the advantage that it can be used to promote the map constructed in Lemma~\ref{lem:Cantormap} to a smooth (area-preserving) diffeomorphism of the disk.
\end{Remark}

Let us briefly explain the main ideas of the construction of the map in Lemma~\ref{lem:Cantormap}. If we fix our attention on a single Cantor block, the piecewise linear map is constructed as the composition of two linear maps. The first one is a translation (depending on the function $G$ of the generalized shift), which sends a block onto another one. Next, using the function $F$, we get an integer which tells us how many shifts have to be applied to the block. The action of the shift map on a block can be obtained by restriction of a positive or negative power of the horseshoe map. This second well known map, is a composition of a translation, a rotation and a rescaling in each coordinate. We finish with the following example, which illustrates Lemma~\ref{lem:Cantormap}.

\begin{Example}
A simple example of a generalized shift and its associated piecewise linear map can be constructed as follows. Consider a generalized shift with alphabet $\{0,1\}$, and such that $D_F=D_G=\{-1,0\}$. We define the functions $F$ and $G$ as: $G(0.1)=0.1$, $G(1.1)=0.0$, $G(0.0)=0.1$, $G(1.0)=1.1$ and $F(0.1)=F(0.0)=-1$, $F(1.1)=F(1.0)=0$. By assigning letters to the Cantor blocks corresponding to each possible finite string of two elements, the associated map can be represented by blocks. Denote by $A,B,C$ and $D$ the Cantor blocks whose corresponding sequences have in positions $-1,0$ respectively the pairs $(0.1),(1.1),(0.0)$ and $(1.0)$; the position of these blocks in the square is computed following the assignment that we introduced before, thus obtaining Figure~\ref{fig:mapsquare} (in the same figure we also represent the images of the blocks). The piecewise linear map can be explicitly written as:

\begin{equation*}
(x,y) \longmapsto \begin{cases} (3x, y/3) \text{ if } (x,y)\in A\\
(3(x-2/3),1/3(y-2/3))\text{ if } (x,y)\in B\\
(x,y+2/3) \text{ if } (x,y)\in C\cup D
\end{cases}
\end{equation*}
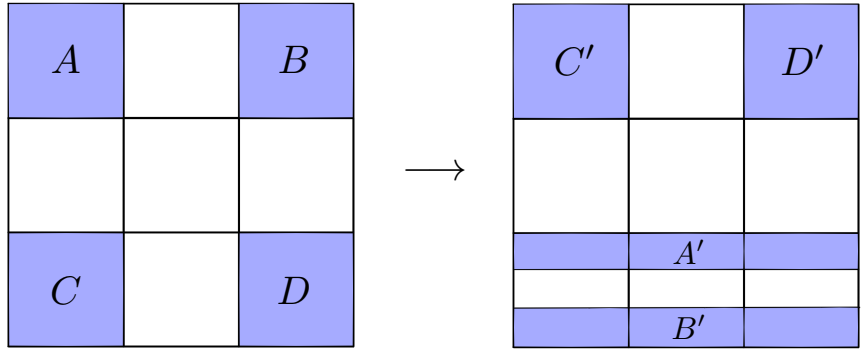
\begin{figure}[!ht]

\begin{center}
\begin{tikzpicture}
     \node[anchor=south west,inner sep=0] at (0,0) {\includegraphics[scale=0.15]{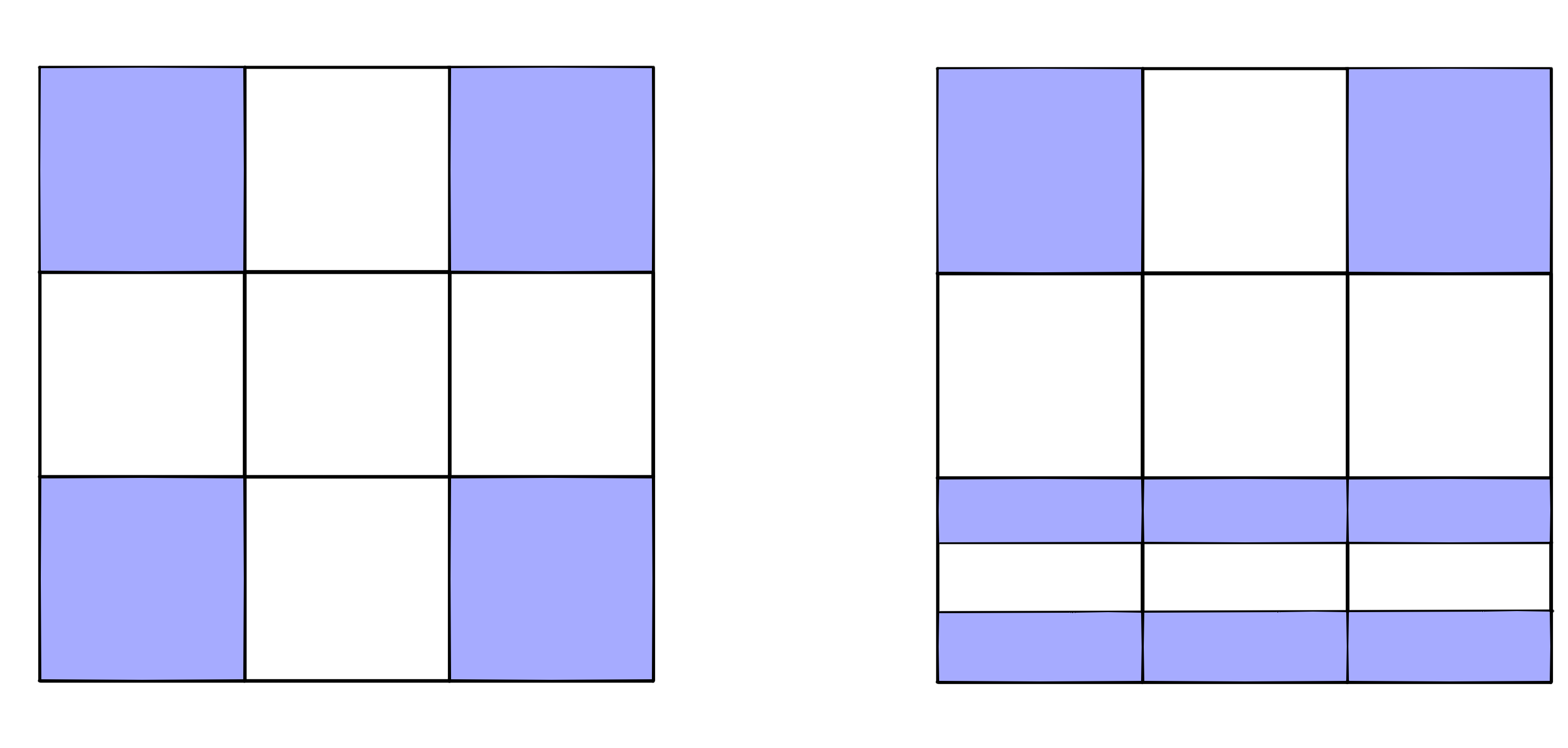}};

\node[scale=1.5] at (1.05,4.3) {$A$};
\node[scale=1.5] at (1.05,1.24) {$C$};
\node[scale=1.5] at (4.05,4.3) {$B$};
\node[scale=1.5] at (4.05,1.24) {$D$};

\node[scale=1.5] at (5.9,2.8) {$\longrightarrow$};

\node[scale=1.5] at (7.75,4.3) {$C'$};
\node[scale=1.5] at (10.75,4.3) {$D'$};
\node[scale=1.2] at (9.25,1.75) {$A'$};
\node[scale=1.2] at (9.25,0.76) {$B'$};
\end{tikzpicture}
\caption{Blocks map in the unit square}
\label{fig:mapsquare}
\end{center}
\end{figure}

\end{Example}

\section{An area-preserving diffeomorphism of the disk that is Turing complete}\label{S:diffeo}

The goal of this section is to construct an area-preserving diffeomorphism of the disk that simulates a universal Turing machine. The main tool is the generalized shifts introduced in Section~\ref{SS.genshift} and their connection with piecewise linear maps of Cantor blocks. In this direction, a first simple observation is that if we choose a set of disjoint blocks containing all the Cantor set, they lie in the unit square with some gaps in between that do not contain points of the square Cantor set. We will use these gaps to extend the piecewise linear map constructed in Lemma~\ref{lem:Cantormap} to an area-preserving diffeomorphism of the disk, provided that the generalized shift is bijective.

\subsection{Smoothing the map}
For any generalized shift, Lemma~\ref{lem:Cantormap} establishes the existence of a piecewise linear map defined on finitely many Cantor blocks whose action on the square Cantor set is conjugated to the generalized shift. In~\cite[Theorem 12]{Mo}, Moore sketches an argument to extend this map to a diffeomorphism of the disk. In the following proposition, using standard arguments, we show that this map can be done area-preserving as long as the generalized shift is bijective.

\begin{prop}\label{prop:GSarea}
For each bijective generalized shift and its associated map of the square Cantor set $\phi$, there exists an area-preserving diffeomorphism of the disk $\varphi:D\rightarrow D$ which is the identity in a neighborhood of $\partial D$ and whose restriction to the square Cantor set is conjugated to $\phi$.
\end{prop}

\begin{proof}

For simplicity we use the same notation $\phi$ for the generalized shift and its associated map of the square Cantor set obtained via Lemma~\ref{lem:Cantormap}. This map is defined on a finite disjoint union of Cantor blocks (that contain the whole square Cantor set), and the images of these blocks are pairwise disjoint because the generalized shift is bijective. Taking an open neighborhood $D$ (diffeomorphic to a disk) of the square $I^2$, our goal is to extend $\phi$ to the whole $D$.

To this end, we start by choosing a set contained in $D$ of disjoint (small enough) open neighborhoods $B_i$ of each Cantor block. Since the map $\phi$, which is piecewise linear, is obviously defined on each neighborhood, it maps each $B_i$ onto a neighborhood $V_i$ of the images of the Cantor blocks. Obviously, $V_i$ are pairwise disjoint and have the same area as $B_i$.

\begin{figure}[!ht]

\begin{center}
\begin{tikzpicture}
     \node[anchor=south west,inner sep=0] at (0,0) {\includegraphics[scale=0.15]{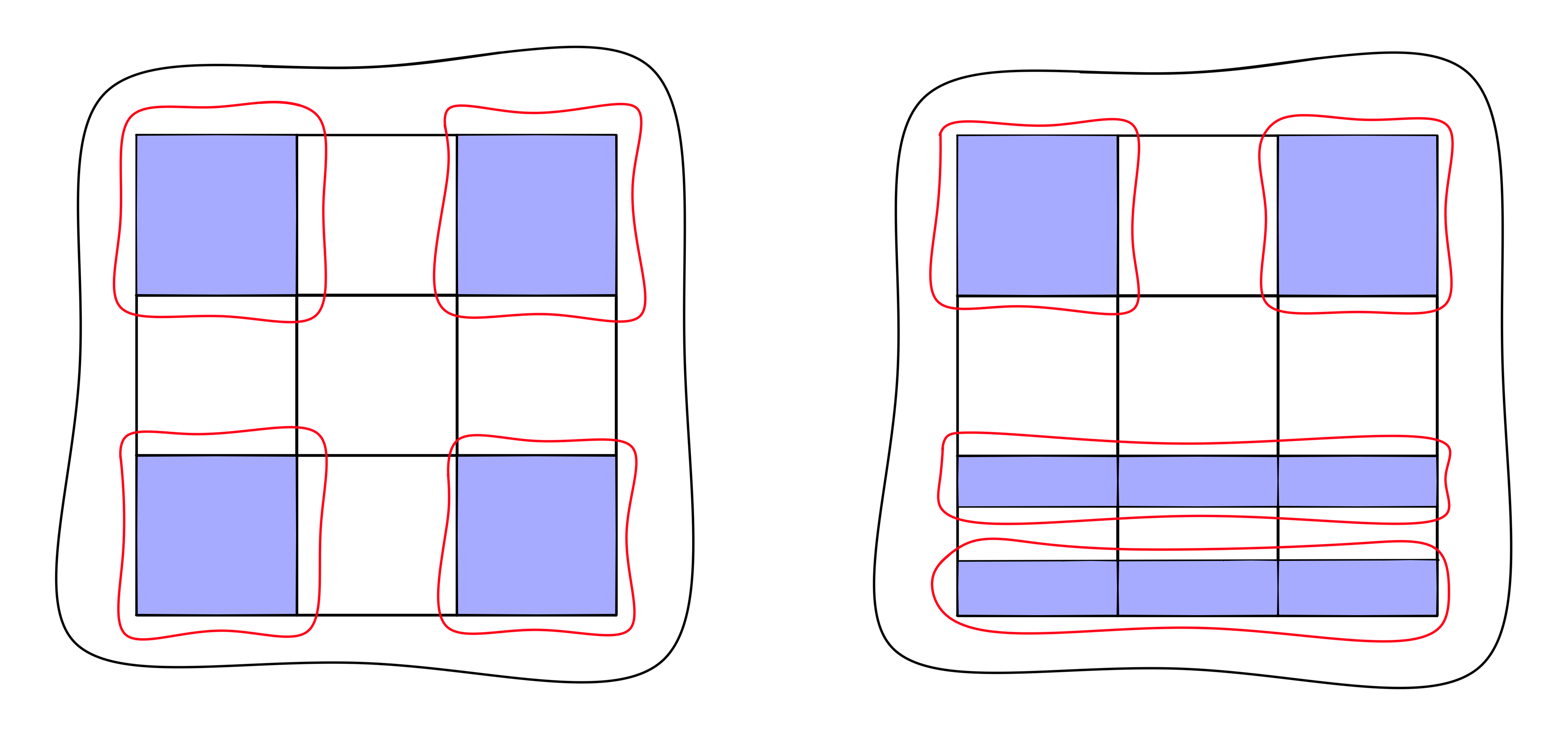}};

\node[scale=1.5] at (2,4.87) {$A$};
\node[scale=1.5] at (2,1.85) {$C$};
\node[scale=1.5] at (5.05,4.87) {$B$};
\node[scale=1.5] at (5.05,1.85) {$D$};

\node[scale=1.5] at (7.5,3.2) {$\longrightarrow$};

\node[scale=1.5] at (9.85,4.87) {$C'$};
\node[scale=1.5] at (12.88,4.87) {$D'$};
\node[scale=1.2] at (11.3,2.35) {$A'$};
\node[scale=1.2] at (11.3,1.35) {$B'$};
\end{tikzpicture}
\caption{Blocks map by open balls}
\label{fig:mapballs}
\end{center}
\end{figure}
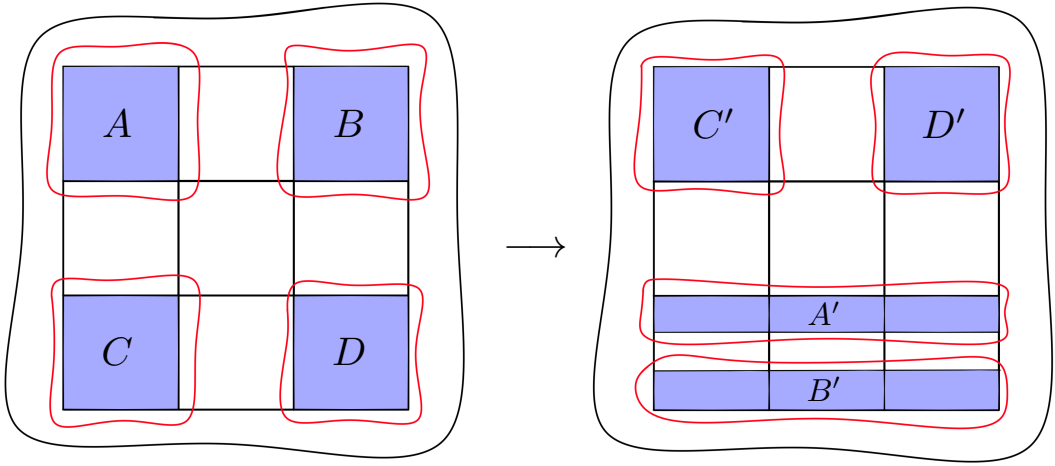

This immediately yields a diffeomorphism $F:\bigcup B_i \rightarrow  \bigcup V_i$ that preserves the standard area form $\omega_{std}=dx\wedge dy$. We claim that this map extends to a diffeomorphism of the disk that is isotopic to the identity. To prove this, we construct a family of maps $F_t:\bigcup B_i \rightarrow D$ such that $F_1=F$, $F_0=\operatorname{id}$, and $F_t$ is a diffeomorphism into its image for each $t\in [0,1]$. To construct this family we first define $F_t^{(1)}$ for $t\in [0,1/3]$ to be an homothety in each open ball $B_i$, which contracts each ball to a ball of small enough area $\delta$. Specifically, taking a point $p_i\in B_i$ such that $B_i$ is star shaped with respect to $p_i$ (this is possible because $B_i$ is a neighborhood of a Cantor block), for each $x\in B_i$

$$F_t^{(1)}(x):=p_i+\lambda_t (x-p_i)\,,$$

where $\lambda_t$ is a smooth function on $B_i$ such that $\lambda_0=1$ and $\lambda_{1/3}<\delta$. Next we choose different points $q_i$ inside the image balls $V_i$ and construct $F_t^{(2)}:\bigcup F_{1/3}^{(1)}(B_i)\rightarrow D$ for $t\in [1/3,2/3]$ to be a map such that $F_{1/3}^{(2)}=\operatorname{id}$ and $F_{2/3}^{(2)}$ sends each ball $F_{1/3}^{(1)}(B_i)$ inside a $\delta$-neighborhood of $q_i$ (contained in $V_i$). Finally, we define an expansion $F_t^{(3)}:\bigcup F_{2/3}^{(2)}(F_{1/3}^{(1)}(B_i))\rightarrow D$ for $t\in [2/3,1]$, analogous to $F_t^{(1)}$, so that $F_{2/3}^{(3)}=\operatorname{id}$ and $F_{1}^{(3)}(F_{2/3}^{(2)}(F_{1/3}^{(1)}(B_i)))=V_i$. The map $F_t$ is then obtained as:

\begin{equation*}
F_t:=\begin{cases}
F_t^{(1)} \text{ for } t\in[0,1/3]\,,\\
F_t^{(2)}\circ F_{1/3}^{(1)} \text{ for } t\in [1/3,2/3]\,,\\
F_t^{(3)}\circ F_{2/3}^{(2)}\circ F_{1/3}^{(1)} \text{ for } t\in [2/3,1]\,.
\end{cases}
\end{equation*}

Now, if we set $F_t$ to be the identity in a neighborhood of the boundary of $D$, the homotopy extension property allows us to extend $F_t$ to a family of diffeomorphisms $\varphi_t$ of the disk such that $\varphi_1|_{\bigcup B_i}=F$, $\varphi_0=\operatorname{id}$ and $\varphi_t$ is the identity near the boundary of $D$ for all $t$.

The standard area form $\omega_{std}$ is sent to another area form that we denote by $\omega_1:= \varphi_1^*\omega_{std}$. Notice that $\omega_1=\omega_{std}$ on $\bigcup B_i$ because $F$ is area-preserving. Additionally, we can interpolate linearly between these two area forms:
$$ \omega_t:=t \omega_1 + (1-t) \omega_{std}\,.$$
Of course, $\omega_t$ is nondegenerate for all $t\in[0,1]$. Noticing that both forms have the same area since $\int_D \omega_{std}=\int_{\varphi^*D} \omega_{std}=\int_D \varphi^*\omega_{std}=\int_D \omega_1$, it follows that $\omega_1-\omega_{std}$ is an exact $2$-form. Applying Moser's path method we then obtain a family of diffeomorphisms $G_t: D\rightarrow D$, $G_0=\operatorname{id}$, such that $G_t^*\omega_t=\omega_{std}$ for all $t\in[0,1]$. Moreover, we can assume that $G_t|_{\bigsqcup B_i}=id$ because $\omega_t=\omega_{std}$ for all $t\in[0,1]$. Finally, the diffeomorphism $\varphi:=\varphi_1 \circ G_1$ satisfies the required conditions, i.e., $\varphi|_{\bigcup B_i}=F$ and $\varphi^*\omega_{std}=\omega_{std}$. The proposition follows noticing that $D$ can be identified with the unit disk in $\mathbb R^2$, after applying a suitable diffeomorphism.
\end{proof}

\subsection{A Turing complete area-preserving diffeomorphism of the disk} \hfill

We are now ready to establish the existence of a Turing complete area-preserving diffeomorphism of the disk that is the identity on the boundary. We remark that our notion of Turing completeness is slightly different from the one used in~\cite{T1, CMPP}, see Remark~\ref{R:Tao} below, but it has the same computational power. Key to the proof are Proposition~\ref{prop:GSarea} and the constructibility of the unique point in the square Cantor associated to an infinite sequence in $\{0,1\}^\Z$, cf. Section~\ref{SS.genshift}. In the proof we also make use of an instrumental result (Lemma~\ref{lem:TuringReads}) allowing us to show that our area-preserving diffeomorphism can check a finite substring of the output of a Turing machine.

%To simplify,  Given an infinite sequence $s=(...t_{-1}.t_0t_1...)$, we can associate to it an explicitely constructible point in the square Cantor set. The classical way to this is to express the coordinates of the assigned point in base $3$: its coordinate $y$ corresponds to the expansion $(y_1,y_2,...)$ where $y_i=0$ if $t_i=0$ and $y_i=2$ if $t_i=1$. Analogously, the coordinate $x$ corresponds to the expansion $(x_{1},x_{2},...)$ in base $3$ where $x_i=0$ if $t_{-i}=0$ and $x_i=2$ if $t_{-i}=1$. This will imply the constructibility of the point in the incoming Theorem.

%{\red La discusion que viene ahora queda un poco rara antes de demostrar la existencia del area-preserving map que simula la maquina de Turing. Tal vez seria mejor meter esto y el Lemma 3.7 dentro de la demostración del Teorema 3.8.}Using the previous Lemma and the constructibility of the point in the associated disk map of a generalized shift

\begin{theorem}\label{thm:areaTuring}
There exists a Turing complete area-preserving diffeomorphism $\varphi$ of the disk that is the identity in a neighborhood of the boundary. Specifically, for any integer $k\geq 0$, given a Turing machine $T=(Q,q_0,q_{halt},\Sigma,\delta)$, an input tape $(t_n) \in \Sigma^{\mathbb{Z}}$ and a finite string $(t_{-k}^*,...,t_k^*)\in \Sigma^{2k+1}$, there is an explicitly constructible point $p\in D$ and an explicitly constructible open set $U\subset D$ such that the orbit of $\varphi$ through $p$ intersects $U$ if and only if $T$ halts with an output tape whose positions $-k,...,k$ correspond to $t_{-k}^*,...,t_k^*$.
\end{theorem}

\begin{proof}
The first observation is that there are several constructions of reversible universal Turing machines. For instance, in~\cite{MY} there is an explicit construction with $17$ states and an alphabet of $5$ symbols. In fact, it is known~\cite{Be} that for any Turing machine (and in particular for a universal one) there is a reversible Turing machine doing the same computations. Hence, let us denote by $T_{un}$ some reversible universal Turing machine. By Lemma~\ref{lem:TuringGS}, we can associate to $T_{un}$ a conjugated generalized shift $\phi$, which is, in fact, bijective in view of Lemma~\ref{lem:revTur}. Applying Proposition \ref{prop:GSarea}, we can construct an area-preserving diffeomorphism $\varphi$ of the disk $D$ which is the identity in a neighborhood of $\partial D$ and whose restriction to the square Cantor set is conjugated to $\phi$.

We claim that the map $\varphi$ is Turing complete. Indeed, given a Turing machine $T=(Q,q_0,q_{halt},\Sigma,\delta)$ and a finite part of the output tape $(t_{-k}^*,...,t_k^*)\in \Sigma^{2k+1}$, Lemma~\ref{lem:TuringReads} below allows us to construct another Turing machine $T'$ which reads the output of $T$. Since $T_{un}$ is universal, it can simulate the evolution of $T'$. In particular, given an input $(q,t)$ of $T'$ there is an explicit input $(\widehat q,\widehat t)$ of $T_{un}$, with $\widehat q \in Q_{un}$ and $\widehat t \in \Sigma_{un}^{\mathbb{Z}}$ (here, $Q_{un}$ and $\Sigma_{un}$ are the space of states and the alphabet of $T_{un}$, respectively), such that $T'$ halts with the aforementioned input if and only if $T_{un}$ halts with input $(\widehat q,\widehat t)$.

As explained in the proof of Lemma~\ref{lem:TuringGS} we obtain an explicit sequence $s=(...s_{-1}.s_0s_1...)\in A^{\Z}$, where $A= \Sigma_{un}\cup Q_{un}$, from the input $(\widehat q,\widehat t)$ of $T_{un}$, which defines a unique point $p\in D$ via the correspondence introduced in Section~\ref{SS.genshift}. For example, taking $\{0,1\}$ to be the alphabet $A$ of the generalized shift $\phi$ (which can always be done as mentioned before), then the coordinates of $p$ are given by

$$p=(x,y)=(\sum_{i=1}^\infty \frac{2s_{-i}}{3^i}, \sum_{i=1}^\infty \frac{2s_{i-1}}{3^i} )\,.$$

(In general, it is an expansion in base $2r+1$ where $r$ is the number of symbols of the alphabet.) Finally, we can take the open set $U$ to be a neighborhood of the Cantor blocks (and hence a disjoint set) that contain all the points associated to sequences in $A^{\Z}$ whose symbols in positions $-1,0,1$ are of the form $(a_{-1}q_{halt}^{un}a_{1})$ for some $a_{\pm 1}\in A$. Here $q_{halt}^{un}$ is the halting state of $T_{un}$. Of course, this set $U$ exists because the Cantor blocks are pairwise disjoint, and it is explicitly constructible. This completes the proof of the theorem.
\end{proof}

\begin{Remark}\label{R:Tao}
%NO ENTIENDO NADA DE ESTE REMARK Observe that the open set $U$ that checks a finite number of positions of the output tape cannot be constructed taking an appropriate open subset of the Cantor blocks whose associated state is $q_{halt}$ and first digits of the base $k$ expansion correspond to $t_{-k}^*,...,t_k^*$ as done in \cite{T1}. This does not work in this construction because of the following discussion. The orbit through an initial point $p$ could intersect a block with state $q_{halt}$ with some output different from $(t_{-k}^*,...,t_k^*)$. The orbit will continue iterating through different states, since to get an invertible map it will go from $q_{halt}$ to $q_0$ and could eventually intersect again a block with halting state, but now with the correct output. This would lead to wrongly detecting the output, since the machine had already stopped previously with another one. Hence the property of halting with given output if and only if it touches that open set would not be satisfied.

There is a key technical difference between the diffeomorphism $\varphi$ we construct in Theorem~\ref{thm:areaTuring} and the Turing complete diffeomorphism of $\mathbb T^4$ constructed in~\cite{T1}. In Tao's construction, the point $p$ depends only on the Turing machine $T$ and the input $(q,t)$. Then, for any given finite string $t^*:=(t_{-k}^*,...,t_k^*)$ there is some open set $U_{t^*}$ such that the orbit through $p$ intersects $U_{t^*}$ if and only if $T$ halts with input $(q,t)$ and output whose positions $-k,...,k$ correspond to $t^*$. In contrast, in the diffeomorphism $\varphi$ we construct in Theorem~\ref{thm:areaTuring}, the point $p$ depends on all the information: the Turing machine $T$, the input $(q,t)$ and the finite string $t^*=(t_{-k}^*,...,t_k^*)$. In particular, if we pick another finite string $t_2^*$, the point $p$ will be different. Additionally, $U$ is always the same, i.e., a neighborhood of those blocks associated to the halting state of $T_{un}$.
\end{Remark}

Finally, we prove the lemma that is used in the proof of Theorem~\ref{thm:areaTuring}: given a Turing machine $T$ and a finite string $(t_{-k}^*,...,t_k^*)$, one can construct a Turing machine $T'$ which halts with a given input if and only if $T$ halts with the same input and with the output tape having in positions $(-k,...,k)$ the fixed symbols $(t_{-k}^*,...,t_k^*)$. This is intuitively clear, one simply needs to construct a machine $T'$ that works exactly as $T$, but when $T$ reaches the halting state, $T'$ reads the positions $-k,...k$ to compare with $(t_{-k}^*,...,t_k^*)$. This is formalized in the following lemma (which is probably standard in the theory of Turing machines).

\begin{lemma}\label{lem:TuringReads}
Let $T=(Q,q_0,q_{halt},\Sigma,\delta)$ be a Turing machine. For any $k\geq 0$ and finite string $(t_{-k}^*,..,t_{k}^*) \in \Sigma^{2k+1}$, there is a Turing machine $T'$ which halts with input $(q_0,t)$ if and only if the machine $T$ with input $(q_0,t)$ halts with coefficients $t_{-k}^*,...,t_k^*$ in positions $-k,...,k$ in the output tape.
\end{lemma}
\begin{proof}
Fix a Turing machine $T=(Q,q_0,q_{halt},\Sigma,\delta)$ and a finite string $(t_{-k}^*,...,t_k^*)$. As before, $F_1,F_2,F_3$ denote the three components of the transition function $\delta: Q\times \Sigma \rightarrow Q\times \Sigma \times \{-1,0,1\} $. To define the Turing machine $T'$, take as alphabet $\Sigma':=\Sigma$ and as set of states $Q':=Q\sqcup \{ r_0,...,r_{3k},q_{nohalt}\}$, where $r_j$ and $q_{nohalt}$ simply denote new states we include in the space. The initial and halting states of $T'$ are the same as for $T$. The idea is to use the states $r_0,...,r_{3k}$ as ``reading states'' that will check if the final output is the desired one.

Let us denote the current state of the Turing machine by $(q,t)$ and by $t_0$ the symbol in the central position. The transition function $\delta'$ can be defined as follows. For $q\in Q\setminus \{q_{halt}\}$, if $F_1(q, t_0)\in Q\setminus \{q_{halt}\}$ then we set $\delta'(q,t_0):=\delta(q,t_0)$ and if $F_1(q,t_0)=q_{halt}$ we define $\delta'(q,\tilde t):=(r_0,F_2(q,t_0),F_3(q,t_0))$. This way, when $T$ reaches a halting state, $T'$ will reach the state $r_0$.

Now we define the transition function $\delta'$ for $q\in \{r_0,...,r_{3k},q_{nohalt}\}$ as:

\begin{equation}\label{eq:r1}
\delta'(r_i,t_0):=
\begin{cases}
(r_{i+1},t_0, -1) \text{ if } t_0=t_{-i}^* \\
(q_{nohalt},t_0,0) \text{ otherwise}
\end{cases}, \text{ for }i=0,...,k-1\,.
\end{equation}

\begin{equation}\label{eq:r2}
\delta'(r_i,t_0):=
\begin{cases}
(r_{i+1},t_0, +1) \text{ if } t_0=t_{i-2k}^* \\
(q_{nohalt},t_0,0) \text{ otherwise}
\end{cases}, \text{for }i=k,...,3k-1\,,
\end{equation}

and $\delta'(r_{3k},t_0 ):=(q_{halt}, t_0,0)$ if $t_0=t_{k}^*$ and $(q_{nohalt},t_0,0)$ otherwise. Finally, we define $\delta'$ for $q_{nohalt}$ so that the machine gets trapped in a loop, e.g. we can set $\delta'(q_{nohalt}, t_0):=(q_{nohalt}, t_0,0)$ for any symbol $t_0$.

Let us check that $T'$ satisfies the required property. Suppose that $T$ halts with a given input $(q_0,t)$. Denote by $t^h:=(...t_{-1}^h.t_0^ht_1^h...)$ the output tape of $T$, i.e., the tape when $T$ reaches the halting state. By the construction, the machine $T'$ with input $(q_0,t)$ will reach the state $r_0$ with tape $t^h$ instead of halting. By Equation~\eqref{eq:r1}, if $t_0^h=t_0^*$ the machine will shift the tape to the right and change the current state to $r_1$. If the symbol $t_0^h$ does not correspond to $t_0^*$, then $T$ enters a loop through the state $q_{nohalt}$ and will never halt.

After shifting to the right, the current tape is now $(...t_{-2}^h.t_{-1}^ht_0^h...)$ and the current state is $r_1$. Again by Equation~\ref{eq:r1}, the machine enters a loop unless $t_{-1}^h=t_{-1}^*$, in which case we shift to the right and change to state $r_2$. Iterating this process, the machine reaches the state $r_k$ if and only if $t_{-i}^h=t_{-i}^*$ for each $i=0,1,...,{k-1}$. The current tape is then $(...t_{-(k+1)}^h.t_{-k}^h...)$. Similarly, by Equation~\eqref{eq:r2} for states $r_j$ with $j=k,...,{3k-1}$, at each step the machine is at the state $r_j$ with current tape $(...t_{j-2k-1}^h.t_{j-2k}^h...)$, and it checks if $t_{j-2k}^h=t_{j-2k}^*$, in which case it shifts to the left with new state $r_{j+1}$. Finally, the machine reaches the state $r_{3k}$ if and only if $(t_{-k}^h....t_{k-1}^h)=(t_{-k}^*...t_{k-1}^*)$, and the current tape becomes $(...t_{k-1}^h.t_{k}^h...)$. By the definition of $\delta'$ at $r_{3k}$, the machine halts if and only if $t_k^h=t_k^*$ or else enters a loop. It is then obvious that $T'$ halts with input $(q_0,t)$ if and only if $T$ halts and its output satisfies that $(t_{-k}^h....t_k^h)=(t_{-k}^*...t_k^*)$, which completes the proof of the lemma.
\end{proof}

\section{Turing completeness of fluid flows}\label{S.Euler}

In this last section we use the Turing complete area-preserving diffeomorphism constructed in Theorem~\ref{thm:areaTuring} to establish the existence of an Eulerisable field in $\mathbb S^3$ which is Turing complete. In the proof we use Etnyre-Ghrist's contact mirror, cf. Theorem~\ref{correspondence}, and the realization Theorem~\ref{prop:Reebdisk} which allows one to embed a diffeomorphism of the disk as the return map of a Reeb flow.

\subsection{Embedding diffeomorphisms as cross sections of Beltrami flows}

In~\cite{CMPP} we constructed a Turing complete Eulerisable flow on $\mathbb S^{17}$ using a new $h$-principle for Reeb embeddings; the dimension $17$ is essentially sharp with this approach. In contrast, the ideas we introduced in Sections~\ref{S:contact},~\ref{S:Turing} and~\ref{S:diffeo} allow us to reduce the dimension to $3$, as shown in the following theorem, which is the main result of this work.

\begin{theorem}\label{thm:main}
There exists an Eulerisable flow $X$ in $\mathbb S^3$ that is Turing complete in the following sense. For any integer $k\geq 0$, given a Turing machine $T$, an input tape $t$, and a finite string $(t_{-k}^*,...,t_k^*)$ of symbols of the alphabet, there exist an explicitly constructible point $p\in \mathbb S^3$ and an open set $U\subset \mathbb S^3$ such that the orbit of $X$ through $p$ intersects $U$ if and only if $T$ halts with an output tape whose positions $-k,...,k$ correspond to the symbols $t_{-k}^*,...,t_k^*$. The metric $g$ that makes $X$ a stationary solution of the Euler equations can be assumed to be the round metric in the complement of an embedded solid torus.
\end{theorem}

\begin{proof}
By~Theorem \ref{thm:areaTuring}, there exists a Turing complete area-preserving diffeomorphism $\varphi$ of the disk which is the identity in a neighborhood of the boundary. Take the standard contact sphere $(\mathbb S^3,\xi_{std})$ and apply Theorem~\ref{prop:Reebdisk} to obtain a defining contact form $\alpha$ whose Reeb field $X$ exhibits an invariant solid torus $\mathcal T$ where the first return map on a disk section is conjugated to $\varphi$ via a diffeomorphism $\Phi:D\to D$. By Remark~\ref{Rem:fixcont} we can assume that in the complement of a neighborhood $V$ of $\overline {\mathcal T}$, the one form $\alpha$ coincides with the standard contact form $\alpha_{std}$ of $\mathbb S^3$. In particular, $X$ coincides with a Hopf field in $\mathbb S^3\backslash V$. When applying the contact/Beltrami correspondence in Lemma~\ref{correspondence}, we obtain a Riemannian metric $g$ which coincides with the round one (as done also in~\cite{CMPP}) on $\mathbb S^3\backslash V$. By construction of the metric, $X$ satisfies the equation $\operatorname{curl}_g X = X$, so it is a stationary solution of the Euler equations on $(\mathbb S^3,g)$.

Finally, let us check that $X$ satisfies the stated Turing completeness property. Take a Turing machine $T$ with an input $t$ and a finite string $(t_{-k}^*,...,t_k^*)$ of symbols. Denoting by $D_0= \{0\} \times D\subset \mathbb S^3$ the transverse section in $\mathcal T$ where the first return map of $X$ is conjugated to $\varphi$, we can find the point $p\in D_0$ and the set $U_0\subset D_0$ (open as a subset of $D$) defined as $p:=\Phi(p_*)$ and $U_0:=\Phi(U_*)$, where $p_*$ and $U_*$ are, respectively, the point and open set given by Theorem~\ref{thm:areaTuring}. We then take the open set
$$U:=\bigcup_{t\in(-\varepsilon_0,\varepsilon_0)}\phi_t(U_0)\,,$$
where $\varepsilon_0>0$ is a small enough constant, and $\phi_t$ is the flow defined by $X$. It is then clear that the point $p\in \mathbb S^3$ and the set $U\subset \mathbb S^3$ satisfy that $T$ will halt with the given output positions if and only if the orbit of $X$ through $p$ intersects the open set $U$, thus completing the proof of the theorem.
\end{proof}

\subsection{Final remark: the Navier-Stokes equations}
%blabla (QUE SE ESPECULA AQUI? NO CLARO SI SE QUIERE SER MAS O MENOS PRECISO)
%
%\paragraph{Relation between entropy and integrability, existence of horseshoe and Katok's conjecture}
%
%Limit of integrable system has zero entropy. (Katok's conjecture) any system with zero entropy can be seen as a limit of integrable system.

The Navier-Stokes equations describe the dynamics of an incompressible fluid flow with viscosity. On a Riemannian $3$-manifold $(M,g)$ they read as
\begin{equation}\label{eq:NS}
\begin{cases}
\pp{u}{t}+ \nabla_u u-\nu \Delta u=-\nabla p\,,\\
\operatorname{div} u=0\,, \\
u(t=0)=u_0\,,
\end{cases}
\end{equation}
where $\nu>0$ is the viscosity. Here all the differential operators are computed with respect to the metric $g$, and $\Delta$ is the Hodge Laplacian (whose action on a vector field is defined as $\Delta u:=(\Delta u^\flat)^\sharp$).

In this final section we analyze what happens with the vector field $X$ constructed in Theorem~\ref{thm:main} when taken as initial condition for the Navier-Stokes equations with the metric $g$ that makes $X$ a steady Euler flow. Specifically, using that $\operatorname{curl}_g(X)=X$, the solution to Equation~\eqref{eq:NS} with $u_0=M X$, $M>0$ a real constant, is easily seen to be
\begin{equation}
\begin{cases}
u(\cdot,t)=MX(\cdot)e^{-\nu t}\,, \\
p(\cdot,t)= c_0 - \frac{1}{2}M^2e^{-2\nu t} \norm{X}^2_g\,,
\end{cases}
\end{equation}
for any constant $c_0$. The integral curves (fluid particle paths) of the non-autonomous field $u$ solve the ODE
$$\frac{dx(t)}{dt}=M e^{-\nu t}X(x(t))\,.$$
Accordingly, reparametrizing the time as
$$\tau(t):=\frac{M}{\nu}(1-e^{-\nu t})\,,$$
we show that the solution $x(t)$ can be written in terms of the solution $y(\tau)$ of the ODE
\[
\frac{dy(\tau)}{d\tau}=X(y(\tau))\,,
\]
as
\[
x(t)=y(\tau(t))\,.
\]

When $t\rightarrow \infty$ the new ``time'' $\tau$ tends to $\frac{M}{\nu}$, and hence the integral curve $x(t)$ of the Navier-Stokes equations travels the orbit of $X$ just for the time interval $\tau\in [0,\frac{M}{\nu})$. In particular, the flow of the solution $u$ only simulates a finite number of steps of a given Turing machine, so we cannot deduce the Turing completeness of the Navier-Stokes equations using the vector field $MX$ as initial condition. More number of steps of a Turing machine can be simulated if $\nu\to 0$ (the vanishing viscosity limit) or $M\to \infty$ (the $L^2$ norm of the initial datum blows up). For example, to obtain a universal Turing simulation we can take a family $\{M_k X\}_{k\in\mathbb N}$ of initial data for the Navier-Stokes equations, where $M_k\rightarrow \infty$ is a sequence of positive numbers. The energy ($L^2$ norm) of this family is not uniformly bounded, thus raising the challenging question of whether there exists an initial datum of finite energy that gives rise to a Turing complete solution of the Navier-Stokes equations.

\end{document}